\documentclass{amsart}

\usepackage{amsmath}
\usepackage{amssymb}
\usepackage{xcolor}

\allowdisplaybreaks

\newtheorem{thm}{Theorem}[section]
\newtheorem{cor}[thm]{Corollary}
\newtheorem{lem}[thm]{Lemma}
\newtheorem{propo}[thm]{Proposition}

\theoremstyle{definition}

\theoremstyle{remark}
\newtheorem{rem}{Remark}

\title[The heat semigroup]{Discrete harmonic analysis \\associated with Jacobi expansions I: \\the heat semigroup}

\author[A. Arenas]{Alberto Arenas}
\address{Departamento de Matem\'aticas y Computaci\'on,
Universidad de La Rioja, Complejo Cient\'{\i}fico-Tecnol\'ogico,
Calle Madre de Dios 53, 26006 Logro\~no, Spain}
\email{alberto.arenas@unirioja.es}

\author[\'O. Ciaurri]{\'Oscar Ciaurri}
\address{Departamento de Matem\'aticas y Computaci\'on,
Universidad de La Rioja, Complejo Cient\'{\i}fico-Tecnol\'ogico,
Calle Madre de Dios 53, 26006 Logro\~no, Spain}
\email{oscar.ciaurri@unirioja.es}

\author[E. Labarga]{Edgar Labarga}
\address{Departamento de Matem\'aticas y Computaci\'on,
Universidad de La Rioja, Complejo Cient\'{\i}fico-Tecnol\'ogico,
Calle Madre de Dios 53, 26006 Logro\~no, Spain}
\email{edgar.labarga@unirioja.es}

\thanks{The first-named author was supported by a predoctoral research grant of the Government of Comunidad Aut\'{o}noma de La Rioja. The second-named author was supported by grant MTM2015-65888-C04-4-P MINECO/FEDER, UE, from Spanish Government. The third-named author was supported by a predoctoral research grant of the University of La Rioja.}

\begin{document}

\begin{abstract}
In this paper we commence the study of discrete harmonic analysis associated with Jacobi orthogonal polynomials of order $(\alpha,\beta)$. Particularly, we give the solution $W^{(\alpha,\beta)}_t$, $t\ge 0$, and some properties of the heat equation related to the operator $J^{(\alpha,\beta)}-I$, where $J^{(\alpha,\beta)}$ is the three-term recurrence relation for the normalized Jacobi polynomials and $I$ is the identity operator. These results will be a consequence of a much more general theorem concerning the solution of the heat equation for Jacobi matrices. In addition, we also prove the positivity of the operator $W^{(\alpha,\beta)}_t$ under some suitable restrictions on the parameters $\alpha$ and $\beta$. Finally, we investigate mapping properties of the maximal operators defined by the heat and Poisson semigroups in weighted $\ell^{p}$-spaces using discrete vector-valued local Calder\'{o}n-Zygmund theory. For the Poisson semigroup, these properties follows readily from the control in terms of the heat one.
\end{abstract}

\maketitle

\section{Introduction}

For $\alpha,\beta>-1$ and $n=0,1,2,\dots$, we consider the sequences $\{a_{n}^{(\alpha,\beta)}\}_{n\in\mathbb{N}}$ and $\{b_{n}^{(\alpha,\beta)}\}_{n\in\mathbb{N}}$ whose elements are given by
\[
a_n^{(\alpha,\beta)}=\frac{2}{2n+\alpha+\beta+2}
\sqrt{\frac{(n+1)(n+\alpha+1)(n+\beta+1)(n+\alpha+\beta+1)}{(2n+\alpha+\beta+1)(2n+\alpha+\beta+3)}},\quad n\geq 1,
\]
\[
a_{0}^{(\alpha,\beta)} = \frac{2}{\alpha+\beta+2}\sqrt{\frac{(\alpha+1)(\beta+1)}{(\alpha+\beta+3)}},
\]
\[
b_n^{(\alpha,\beta)}=\frac{\beta^2-\alpha^2}{(2n+\alpha+\beta)(2n+\alpha+\beta+2)},\quad n\geq 1,
\]
and
\[
b_{0}^{(\alpha,\beta)} = \frac{\beta-\alpha}{\alpha+\beta+2}.
\]
Then, for any given sequence $\{f(n)\}_{n\ge 0}$, we define $\{J^{(\alpha,\beta)}f(n)\}_{n\ge 0}$ by the relations
\[
J^{(\alpha,\beta)}f(n)=a_{n-1}^{(\alpha,\beta)}f(n-1)+b_n^{(\alpha,\beta)}f(n)+ a_{n}^{(\alpha,\beta)}f(n+1), \qquad n\ge 1,
\]
and $J^{(\alpha,\beta)}f(0)=b_0^{(\alpha,\beta)}f(0)+ a_{0}^{(\alpha,\beta)}f(1)$.

The sequences $\{a_n^{(\alpha,\beta)}\}_{n\ge 0}$ and $\{b_n^{(\alpha,\beta)}\}_{n\ge 0}$ are the ones involved in the three-term recurrence relation for the normalized Jacobi polynomials. By using the Rodrigues' formula (see \cite[p.~67, eq.~(4.3.1)]{Szego}), the Jacobi polynomials $P^{(\alpha,\beta)}_n(x)$, $n\ge 0$, are defined as
\[
(1-x)^{\alpha}(1+x)^{\beta}P_n^{(\alpha,\beta)}(x)=\frac{(-1)^n}{2^n \, n!}\frac{d^n}{dx^n}\left((1-x)^{\alpha+n}(1+x)^{\beta+n}\right).
\]
They are orthogonal on the interval $[-1,1]$ with respect to the measure
\[
d\mu_{\alpha,\beta}(x)=(1-x)^\alpha(1+x)^{\beta}\,dx.
\]
The family $\{p_n^{(\alpha,\beta)}(x)\}_{n\ge 0}$, given by $p_n^{(\alpha,\beta)}(x)=w_n^{(\alpha,\beta)}P_n^{(\alpha,\beta)}(x)$, where
\begin{equation*}
\begin{aligned}
w_n^{(\alpha,\beta)}& = \frac{1}{\|P_n^{(\alpha,\beta)}\|_{L^2([-1,1],d\mu_{\alpha,\beta})}} \\&= \sqrt{\frac{(2n+\alpha+\beta+1)\, n!\,\Gamma(n+\alpha+\beta+1)}{2^{\alpha+\beta+1}\Gamma(n+\alpha+1)\,\Gamma(n+\beta+1)}},\quad n\geq1,
\end{aligned}
\end{equation*}
and
\[
w_{0}^{(\alpha,\beta)} = \frac{1}{\|P_{0}^{(\alpha,\beta)}\|_{L^2([-1,1],d\mu_{\alpha,\beta})}} = \sqrt{\frac{\Gamma(\alpha+\beta+2)}{2^{\alpha+\beta+1}\Gamma(\alpha+1)\Gamma(\beta+1)}},
\]
is a complete orthonormal system in the space $L^2([-1,1],d\mu_{\alpha,\beta})$. Furthermore, we have that
\[
J^{(\alpha,\beta)}p^{(\alpha,\beta)}_n(x)=xp_n^{(\alpha,\beta)}(x),\qquad x\in\ [-1,1],
\]
which is obtained by a straightforward computation from the three-term recurrence relation for the orthogonal Jacobi polynomials, namely \cite[eq. 4.5.1, p.71]{Szego}
\begin{align*}
2n(n+&\alpha+\beta)(2n+\alpha+\beta-2)P_{n}^{(\alpha,\beta)}(x) \\
	&=(2n+\alpha+\beta-1)\left((2n+\alpha+\beta)(2n+\alpha+\beta-2)x+\alpha^2-\beta^2\right)P_{n-1}^{(\alpha,\beta)}(x) \\
	&\phantom{=}-2(n+\alpha-1)(n+\beta-1)(2n+\alpha+\beta)P_{n-2}^{(\alpha,\beta)}(x)
\end{align*}
for $n\geq2$, with
\[
P_{0}^{(\alpha,\beta)}(x)=1, \quad\text{and}\quad
P_{1}^{(\alpha,\beta)}(x)=\frac{\alpha+\beta+2}{2}x+\frac{\alpha-\beta}{2}.
\]

The main purpose of this paper is the study of discrete harmonic analysis associated with the operator $J^{(\alpha,\beta)}$. Therefore, this work can be viewed as a generalization of the results presented in \cite{Ciau-et-al} for the discrete Laplacian
\begin{equation}
\label{ec:discrete}
\Delta_d f(n)=f(n-1)-2f(n)+f(n+1),\quad n\in\mathbb{Z},
\end{equation}
and in \cite{Bet-et-al} for ultraspherical expansions, which corresponds with the case of $J^{(\alpha,\beta)}$ $\alpha=\beta=\lambda-1/2$, $\lambda>-1/2$. On the other hand, it can also be considered as the discrete version of the harmonic analysis for the continuous Jacobi expansions developed, for example, in~\cite{AdamSjogren}.

It is convenient for us to work with the operator
\[
\mathcal{J}^{(\alpha,\beta)}f(n)=(J^{(\alpha,\beta)}-I)f(n),
\]
rather than working with $J^{(\alpha,\beta)}$ (here $I$ denotes the identity operator), since the translated operator $-\mathcal{J}^{(\alpha,\beta)}$ is non-negative. In fact, the spectrum of $J^{(\alpha,\beta)}$ is the interval $[-1,1]$, so that the spectrum of $-\mathcal{J}^{(\alpha,\beta)}$ is $[0,2]$. Observe that one could also get a positive operator by defining $\mathcal{\tilde{J}}^{(\alpha,\beta)}f(n)=(J^{(\alpha,\beta)}+I)f(n),$ where in this case the spectrum would be  the interval $[0,2]$ and similar results of this paper would be attained.

Our main goal here is to give a precise and simple expression of the heat semigroup associated with $\mathcal{J}^{(\alpha,\beta)}$ and to establish some of its properties. In forthcoming papers, we will analyse other typical aspects and operators (Riesz transform, weighted inequalities for potential operators, Littlewood-Paley functions, etc.) of harmonic analysis related to Jacobi expansions in the discrete setting.

For $n\geq 0$, $\alpha$, $\beta>-1$, and $t\ge 0$, the heat equation in this context is given
\[
\frac{\partial u(n,t)}{\partial t}=\mathcal{J}^{(\alpha,\beta)}u(n,t)
\]
and for each sequence $\{f(n)\}_{n\ge 0}$ the corresponding initial-value problem is
\begin{equation}\label{initialvalueproblem}
\begin{cases}
\dfrac{\partial u(n,t)}{\partial t}=\mathcal{J}^{(\alpha,\beta)}u(n,t),\\[4pt]
u(n,0)=f(n).
\end{cases}
\end{equation}
We will show in Section~\ref{sec:semi} that a solution is $u(n,t)=W_t^{(\alpha,\beta)}f(n)$, with
\[
W_t^{(\alpha,\beta)}f(n)=\sum_{m=0}^\infty f(m)K^{(\alpha,\beta)}_t(m,n),
\]
where the kernel is
\[
K_t^{(\alpha,\beta)}(m,n)=\int_{-1}^1 e^{-(1-x)t}p_m^{(\alpha,\beta)}(x)p_n^{(\alpha,\beta)}(x)\, d\mu_{\alpha,\beta}(x).
\]
Additionally, as a consequence of a much more general result concerning the solution of the heat equation for Jacobi matrices, we will prove that the family $\{W_t^{(\alpha,\beta)}\}_{t\ge 0}$ is a strongly continuous semigroup of operators on $\ell^{2}(\mathbb{N})$ (see Theorem~\ref{thm:semigroup}).

It is worth pointing out that the kernel of the heat semigroup can be written in terms of special functions in the particular case $\alpha=\beta=-1/2$, which is related to Chebyshev polynomials of the first kind. Indeed, the change of variable $x=\cos \theta$, $\theta\in [0,\pi]$, lets us to write
\[
K_t^{(-1/2,-1/2)}(m,n)=\int_0^\pi e^{-t(1-\cos\theta)}\cos(m\theta)\cos(n\theta)\, d\theta.
\]
Then, the identity (see \cite[p. 456]{PBM1})
\[
\frac{1}{\pi}\int_0^\pi e^{z\cos \theta}\cos(m\theta)\, d\theta=I_m(z), \qquad |\arg(z)|<\pi,
\]
where $I_m$ denotes the Bessel function of imaginary argument and order $m$, leads us to
\[
K_t^{(-1/2,-1/2)}(m,n)=\frac{\pi}{2}e^{-t}(I_{n-m}(t)+I_{n+m}(t)).
\]
In \cite{Ciau-et-al}, the starting point to study harmonic analysis for the discrete Laplacian \eqref{ec:discrete} was the study of the heat operator associated to it. The kernel of this operator is, precisely, the special function $I_m$. So, one can consider the present work as a natural extension of the corresponding results in that paper in view of the discussion above. The asymptotic behaviour of the functions $I_m$ was a fundamental tool in \cite{Ciau-et-al}, however, in our case a deeper analysis of the kernel $K_t^{(\alpha,\beta)}$ will be needed in order to obtain the appropriate estimates.

We will continue our analysis of the heat semigroup $\{W_t^{(\alpha,\beta)}\}_{t \ge 0}$ showing that it is the unique solution of the initial-value problem \eqref{initialvalueproblem} in a suitable space. Moreover, after extending the definition of the heat semigroup for sequences in $\ell^\infty(\mathbb{N})$, we will look at its positivity under some restrictions on the parameters $\alpha$ and $\beta$. The main tool to prove our result about the positivity of the coefficients is the linearization formula for Jacobi polynomials due to Gasper \cite{Gasper2,Gasper1}. Both results and their proofs will be included in Section~\ref{sec:Jacobi-pol}.

Finally, we will devote Section~\ref{sec:max} to deal with some weighted inequalities for the maximal operators related to the heat and Poisson semigroups.

A weight on $\mathbb{N}$ will be a strictly positive sequence $w=\{w(n)\}_{n\ge 0}$. We consider the weighted $\ell^{p}$-spaces
\[
\ell^p(\mathbb{N},w)=\left\{f=\{f(n)\}_{n\ge 0}: \|f\|_{\ell^{p}(\mathbb{N},w)}:=\Bigg(\sum_{m=0}^{\infty}|f(m)|^p w(m)\Bigg)^{1/p}<\infty\right\},
\]
$1\le p<\infty$, and the weak weighted $\ell^{1}$-space
\[
\ell^{1,\infty}(\mathbb{N},w)=\left\{f=\{f(n)\}_{n\ge 0}: \|f\|_{\ell^{1,\infty}(\mathbb{N},w)}:=\sup_{t>0}t\sum_{\{m\in \mathbb{N}: |f(m)|>t\}} w(m)<\infty\right\},
\]
and we simply write $\ell^p(\mathbb{N})$ and $\ell^{1,\infty}(\mathbb{N})$ when $w(n)=1$ for all $n\in \mathbb{N}$.

Furthermore, we say that a weight $w(n)$ belongs to the discrete Muckenhoupt $A_p(\mathbb{N})$ class, $1<p<\infty$, provided that
\[
\sup_{\begin{smallmatrix} 0\le n \le m \\ n,m\in \mathbb{N} \end{smallmatrix}} \frac{1}{(m-n+1)^p}\Bigg(\sum_{k=n}^mw(k)\Bigg)\Bigg(\sum_{k=n}^mw(k)^{-1/(p-1)}\Bigg)^{p-1} <\infty,
\]
and that $w(n)$ belongs to the discrete Muckenhoupt $A_1(\mathbb{N})$ class if
\[
\sup_{\begin{smallmatrix} 0\le n \le m \\ n,m\in \mathbb{N} \end{smallmatrix}} \frac{1}{m-n+1}\Bigg(\sum_{k=n}^mw(k)\Bigg)\max_{n\le k \le m}w(k)^{-1} <\infty,
\]
holds.

For sequences $f\in\ell^{\infty}(\mathbb{N})$, we define  the Poisson operator by
\begin{equation}
\label{eq:poisson}
P^{(\alpha,\beta)}_tf(n)=\frac{1}{\sqrt{\pi}}\int_{0}^{\infty}\frac{e^{-u}}{\sqrt{u}}W_{t^2/(4u)}^{(\alpha,\beta)}f(n)\, du,\quad t\ge 0.
\end{equation}
The maximal heat and Poisson operators are given by
\begin{equation}
\label{eq:max-heat}
W^{(\alpha,\beta)}_{\ast}f(n)=\sup_{t>0}|W_t^{(\alpha,\beta)}f(n)|
\end{equation}
and
\begin{equation}
\label{eq:max-Poisson}
P^{(\alpha,\beta)}_{\ast}f(n)=\sup_{t>0}|P_t^{(\alpha,\beta)}f(n)|.
\end{equation}
We will prove that
\[
\|W^{(\alpha,\beta)}_{\ast}f\|_{\ell^p(\mathbb{N},w)}\le \|f\|_{\ell^p(\mathbb{N},w)}
\]
and
\[
\|P^{(\alpha,\beta)}_{\ast}f\|_{\ell^p(\mathbb{N},w)}\le \|f\|_{\ell^p(\mathbb{N},w)},
\]
for $f\in\ell^{p}(\mathbb{N},w)$, $w$ a weight in the discrete $A_p(\mathbb{N})$ class, $1<p<\infty$, and parameters $\alpha$, $\beta\geq-1/2$. Furthermore, we will also give weak type inequalities in the case $p=1$. These estimates will be obtained as a consequence of an appropriate discrete vector-valued local Calder\'on-Zygmund theory developed in \cite{Bet-et-al}. The result concerning these inequalities (Theorem~\ref{th:max}) and a revisited overview of the aforementioned Calder\'on-Zygmund theory will be given in Section~\ref{sec:max}. The corresponding technical results to show the main estimates will be included in Section~\ref{sec:tech}.

Throughout the paper $C$ will denote a positive constant independent of significant quantities and it possibly takes different values in each case.

\section{The heat semigroup for Jacobi matrices}
\label{sec:semi}
In this section we are interested in the heat semigroup related to Jacobi matrices, that is, infinite tridiagonal matrices
\[
J=\begin{pmatrix}
    b_0 & a_0 & 0 & 0 & \cdots \\
    a_0 & b_1 & a_1 & 0 & \cdots \\
    0 & a_1 & b_2 & a_2 & \cdots \\
    0 & 0 & a_2 & b_3 & \cdots \\
    \vdots & \vdots & \vdots & \vdots & \ddots
  \end{pmatrix},
\]
with $a_n>0$ and $b_n\in \mathbb{R}$ for all $n\in\mathbb{N}$. We suppose that $\{a_n\}_{n\ge 0}$ and $\{b_n\}_{n\ge 0}$ are bounded sequences so that $J$ defines a bounded self-adjoint linear operator on $\ell^2(\mathbb{N})$. In this situation, Favard's theorem (see \cite{Favard} and \cite[Ch.~1, Theorem~4.4]{Chihara}) states that each Jacobi matrix corresponds to a spectral measure $\mu$ with a compact support $X$ having an infinite number of points. Moreover, there exists a family of polynomials $\{p_n\}_{n\ge 0}$ orthonormalized in $L^2(X,\mu)$, i.e.,
\[
\int_X p_n(x)p_m(x)\, d\mu(x)=\delta_{nm},
\]
where $\delta_{nm}$ denotes the Kronecker's delta, satisfying the three-term recurrence relation
\[
xp_n(x)=a_{n-1}p_{n-1}(x)+b_{n}p_n(x)+a_{n}p_{n+1}(x),\quad x\in X,
\]
with $p_{-1}(x)=0$, and where the sequences $\{a_n\}_{n\ge 0}$ and $\{b_n\}_{n\ge 0}$ are the entries of the Jacobi matrix $J$ associated with the measure $\mu$.

As it is well-known, the measure $\mu$ related to a Jacobi matrix may not be unique (see \cite[\S~56]{Stieltjes}). However, if it is unique, then the family of orthonormal polynomials $\{p_n\}_{n\ge 0}$ is dense in $L^2(X,d\mu)$ (see \cite[Theorem~2.14]{ShohatTamarkin}). Hence, the Fourier series related to them is convergent in the space $L^2(X,\mu)$ and for each $f\in L^2(X,\mu)$, with the Fourier coefficients given by
\[
c_m(f) = \int_{X}f(t) p_m(t)\, d\mu(t),
\]
the identity
\[
f(x) = \sum_{m=0}^{\infty}c_m(f) p_m(x),
\]
holds in $L^2(X,\mu)$. In order to guarantee the uniqueness of the measure $\mu$, we suppose that $a_n\to a$ and $b_n\to b$ (with both $a$ and $b$ finite), so $X$ is bounded with at most countably many points outside the interval $[b-2a,b+2a]$, with $b\pm 2a$ the limit points of $X$ (see \cite[Ch.~2, Theorem~5.6]{Chihara}).

The main target of this section is to find the solution of the heat equation associated with a perturbation of the Jacobi matrix $J$ which is given in the following way: for each Jacobi matrix, let $s$ be the maximum of the support of the measure $\mu$ and $s^{+}=\max \{s,0\}$, and define the operator
\[
\mathcal{J}=J-s^{+}I,
\]
where $I$ is the infinite identity matrix. Observe that now
\begin{equation}
\label{ec:Jpn}
\mathcal{J}p_n(x)=(x-s^{+})p_n(x),\quad x\in X.
\end{equation}

Then, for $n\geq0$ and $t\ge 0$, and each appropriate sequence $\{f(n)\}_{n\ge 0}$, we consider the initial-value problem corresponding to the heat equation equation associated with the operator $\mathcal{J}$ given by
\[
\begin{cases}
\dfrac{\partial u(n,t)}{\partial t}=\mathcal{J}u(n,t),\\[4pt]
u(n,0)=f(n).
\end{cases}
\]
Let us begin checking that
\[
u(n,t)=W_tf(n),
\]
where
\begin{equation}
\label{ec:Wt}
W_tf(n)=\sum_{m=0}^{\infty} f(m)K_t(m,n)
\end{equation}
and
\[
K_t(m,n)=\int_X e^{(x-s^{+})t}p_m(x)p_n(x)\, d\mu(x),
\]
is a solution of the initial-value problem. First, note that $W_tf$ is well defined for each sequence $f\in \ell^2(\mathbb{N})$. Indeed,
\[
|W_tf(n)|\le \|f\|_{\ell^2(\mathbb{N})}\|K_t(\cdot, n)\|_{\ell^2(\mathbb{N})}
\]
and, by Parseval's identity,
\[
\|K_t(\cdot, n)\|_{\ell^2(\mathbb{N})}=\|e^{(\cdot-s^{+})t}p_n\|_{L^2(X,\mu)}\le\|p_n\|_{L^2(X,\mu)}=1.
\]
Following a similar argument we have that $\frac{\partial}{\partial t}W_tf(n)$ is well defined and
\[
\frac{\partial}{\partial t}W_tf(n)=\sum_{m=0}^{\infty} f(m)\int_X e^{(x-s^{+})t}(x-s^{+})p_m(x)p_n(x)\, d\mu(x).
\]
Then, using \eqref{ec:Jpn}, we get $\frac{\partial}{\partial t}W_tf(n)=\mathcal{J}W_tf(n)$. In addition, by means of the identity
\begin{equation}
\label{ec:delta}
K_0(m,n)=\delta_{mn},
\end{equation}
where $\delta_{mn}$ denotes the Kronecker's delta, it is immediate to see that $W_0f(n)=f(n)$.

The following theorem shows that the family of operators $\{W_t\}_{t \ge 0}$ is a strongly continuous semigroup of operators on $\ell^2(\mathbb{N})$, i. e. a $C_0$-semigroup, and that $W_t$ satisfies a contraction property.
\begin{thm}
\label{thm:semigroup}
Let $W_t$ the operator defined by \eqref{ec:Wt}, then for each sequence $f\in \ell^2(\mathbb{N})$ we have that
\begin{equation}
\label{ec:l2-bound}
\|W_tf\|_{\ell^2(\mathbb{N})}\le  \|f\|_{\ell^2(\mathbb{N})}.
\end{equation}
Moreover,
\begin{enumerate}
  \item[a)] $W_{t_{1}}W_{t_{2}} f(n)=W_{t_{1}+t_{2}}f(n)$, for $t_{1},t_{2}\ge 0$, $n\in\mathbb{N}$,
  \item[b)] $W_0f(n)=f(n)$, $n\in\mathbb{N}$ and,
  \item[c)] $\lim_{t\to 0^+}\|W_tf-f\|_{\ell^2(\mathbb{N})}=0.$
\end{enumerate}
\end{thm}

Before giving the proof of the theorem we make some important remarks.

It is well-known that for each sequence in $f\in\ell^2(\mathbb{N})$ there is a function $F\in L^2(X,d\mu)$ such that
\begin{equation}
\label{ec:inverse}
F(x)=\sum_{m=0}^\infty f(m)p_m(x),
\end{equation}
where the convergence holds in the $L^2(X,d\mu)$ sense, and $c_m(F)=f(m)$. Furthermore, we have Parseval's identity
\[
\|F\|_{L^2(X,d\mu)}=\|f\|_{\ell^2(\mathbb{N})}.
\]
Obviously, given two sequences $f_1,f_2\in \ell^2(\mathbb{N})$, the identity
\begin{equation}\label{ec:Parseval}
\int_X F_1(x)F_2(x)\, d\mu(x)=\sum_{m=0}^\infty f_1(m)f_2(m),
\end{equation}
where the functions $F_1$ and $F_2$ are defined as in \eqref{ec:inverse}, holds.

\begin{proof}[Proof of Theorem \ref{thm:semigroup}]
Let us first show that \eqref{ec:l2-bound} holds. Note that
\[
W_tf(n)=\int_X e^{(x-s^{+})t}F(x)p_n(x)\, d\mu(x)=c_n(e^{(\cdot-s^{+})t}F)
\]
where $F$ is defined by \eqref{ec:inverse}. So, taking $\ell^2$-norm and considering Parseval's identity, \eqref{ec:l2-bound} is proved. Indeed,
\begin{align*}
\|W_tf\|_{\ell^2(\mathbb{N})}&=\|c_n(e^{(\cdot-s^{+})t}F)\|_{\ell^2(\mathbb{N})}=\|e^{(\cdot-s^{+})t}F\|_{L^2(X,d\mu)}
\\&\le \|F\|_{L^2(X,d\mu)}=\|f\|_{\ell^2(\mathbb{N})}.
\end{align*}

To prove a), we note that
\[
K_{t_{1}}(k,n)=c_k(e^{t_{1}(\cdot-s^{+})}p_n),
\]
and then we use \eqref{ec:Parseval} to get
\begin{equation*}
\begin{aligned}
W_{t_{1}}W_{t_{2}} f(n) &= \sum_{k=0}^{\infty} c_{k}(e^{t_{2}(\cdot-s^{+})}F)c_{k}(e^{t_{1}(\cdot-s^{+})}p_{n})\\ &= \int_{X} e^{(t_{1}+t_{2})(x-s^{+})} F(x) p_{n}(x)d\mu(x) = W_{t_{1}+t_{2}} f(n).
\end{aligned}
\end{equation*}

Part b) is a simple consequence of \eqref{ec:delta}. Finally, to prove part c), we observe that
\[
\|W_tf-f\|_{\ell^2(\mathbb{N})}^2=\|c_n((e^{(\cdot-s^{+})t}-1)F)\|_{\ell^2(\mathbb{N})}^2=\int_{X}(e^{(x-s^{+})t}-1)^2F^2(x)\, d\mu(x)
\]
and the result is immediate.
\end{proof}

A brief comment is in order at this point. It is clear that $W_tf=e^{t\mathcal{J}}f$ by construction, where
\[
e^{t\mathcal{J}}f=\int_{X} e^{-t(s^+-\lambda)}\, dE_{J}(\lambda)f
\]
and $E_J$ is the spectral measure associated with $J$, so $W_t$ is clearly a $C_0$-semigroup. Therefore the properties in Theorem \ref{thm:semigroup} could be obtained as an immediate consequence of this fact. However, our proof exhibits in an explicit way the relation between the heat semigroup $W_t$ and the Fourier expansions in terms of the polynomials $p_n$. This relation will play a fundamental role throughout this paper.

\begin{rem}
Defining the maximal operator by
\[
W_\ast f(n)=\sup_{t>0}|W_tf(n)|, \qquad f\in \ell^2(\mathbb{N}),
\]
using ideas from \cite[Ch. III]{Stein1} and the estimate \eqref{ec:l2-bound}, we can conclude the $\ell^{2}$-boundedness of $W_{\ast}$, that is,
\begin{equation}
\label{ec:L2-max}
\|W_\ast f\|_{\ell^2(\mathbb{N})}\le C \|f\|_{\ell^2(\mathbb{N})}.
\end{equation}
The sketch of the proof is as follows. Following \cite[pp. 74-75]{Stein1}
$$
W_{\ast}f(n)\leq Mf(n)+g(f)(n),
$$
where $g(f)(n)$ denotes the Littlewood-Paley function and $Mf(n)$ is the supremum of the averages of the heat semigroup. They are respectively given by
$$
g(f)(n)=\left(\int_{0}^{\infty}\left|t\frac{\partial}{\partial t}W_{t}f(n)\right|^2\,\frac{dt}{t}\right)^{1/2}
\quad\text{and}\quad
Mf(n)=\sup_{s>0}\left|\frac{1}{s}\int_{0}^{s}W_{t}f(n)\,dt\right|.
$$
Both operators are bounded from $\ell^{2}(\mathbb{N})$ to itself, that is,
$$
\|g(f)\|_{\ell^{2}(\mathbb{N})}\leq C\|f\|_{\ell^{2}(\mathbb{N})}
\quad\text{and}\quad
\|Mf\|_{\ell^{2}(\mathbb{N})}\leq C\|f\|_{\ell^{2}(\mathbb{N})},
$$
so \eqref{ec:L2-max} follows directly. The proof for the bound for $g(f)$ is akin to the spectral argument given in \cite[p. 74]{Stein1}. On its behalf, the bound for $Mf$ is the discrete analogous of the continuous one presented in \cite[Corollary 2]{Stein1961} and it can be proved in a similar way. The contractivity of the semigroup $\|W_{t}\|_{2}\leq 1$, obtained by means of \eqref{ec:l2-bound}, is a requirement there.

Note that an immediate corollary of this fact is the a.e.-convergence of the heat operator, that is,
\[
\lim_{t\to 0^+}W_tf(n)=f(n),\qquad \text{ a.e.,}
\]
for each sequence $f\in \ell^{2}(\mathbb{N})$. As it is well known, in the $\ell^p(\mathbb{N})$ spaces the a.e.-convergence is equivalent to the pointwise convergent. Then,
\begin{equation}
\label{eq:point-conv}
\lim_{t\to 0^+}W_tf(n)=f(n), \qquad n\in \mathbb{N},
\end{equation}
for each $f\in \ell^2(\mathbb{N})$. It is easy to check that \eqref{eq:point-conv} also follows from c) in Theorem~\ref{thm:semigroup}.
\end{rem}

\begin{rem}
From a) in Theorem \ref{thm:semigroup}, it is immediate to get a Chapman-Kolgomorov type identity in our setting. In fact, it is verified that
\[
\sum_{m=0}^{\infty} K_{t_1}(m,n)K_{t_2}(m,j)=K_{t_1+t_2}(n,j),\qquad n,j\ge 0.
\]
\end{rem}

\section{Two additional properties for the heat semigroup $W_{t}^{(\alpha,\beta)}$}
\label{sec:Jacobi-pol}

In the previous section we have analyzed the heat semigroup related to a general Jacobi matrix. Unfortunately, to study other deeper properties of this operator we need extra information of the associated family of orthogonal polynomials. In the particular case of the Jacobi polynomials we have a specific knowledge of their behaviour and we can prove some additional results. In fact, in this section, we prove a uniqueness result for the solution of the initial-value problem \eqref{initialvalueproblem} and we establish the positivity of $W_t^{(\alpha,\beta)}$.

\subsection{A uniqueness result} To prove our uniqueness result we employ a well-known energy method which uses a specific decomposition of our operator. Such decomposition stems from the observation that the sequences $\{d_n\}_{n\geq0}$ and $\{e_n\}_{n\geq0}$ defined by $d_{0} = \sqrt{\frac{2(\alpha+1)}{\alpha+\beta+2}}$,
\begin{equation*}
d_{n} = \sqrt{\frac{2(n+\alpha+\beta+1)(n+\alpha+1)}{(2n+\alpha+\beta+1)(2n+\alpha+\beta+2)}},\quad n\geq 1,
\end{equation*}
and
\begin{equation*}
e_{n} = \sqrt{\frac{2(n+\beta+1)(n+1)}{(2n+\alpha+\beta+2)(2n+\alpha+\beta+3)}}\quad n\geq 0,
\end{equation*}
fulfil the relations
\begin{align*}
&a^{(\alpha,\beta)}_{n}=d_{n}e_{n}, & &\hspace{-5em}n\geq0 \\
&b^{(\alpha,\beta)}_{n}=1-d_{n}^{2}-e_{n-1}^{2}, & &\hspace{-5em} n\geq1
\end{align*}
and
$b^{(\alpha,\beta)}_{0}=1-d_{0}^2$. Then, defining the operators
\begin{equation*}
\delta f(n) = d_{n}f(n) - e_{n}f(n+1),\quad n\geq 0,
\end{equation*}
\begin{equation*}
\delta^{\star} f(n) = d_n f(n) - e_{n-1}f(n-1),\quad n\geq 1,
\end{equation*}
and $\delta^{\star} f(0) = d_{0}f(0)$, it is clear that $\mathcal{J}^{(\alpha,\beta)}$ is decomposed as
\begin{equation*}
\mathcal{J}^{(\alpha,\beta)} = -\delta^{\star}\delta.
\end{equation*}
Note that $\delta$ and $\delta^{\star}$ are adjoint operators in $\ell^{2}(\mathbb{N})$.

To formulate our uniqueness result we introduce the following space of sequences
\[
\ell_{\star}^{2} = \left\{u(n,t) : \sup_{t>0}|u(n,t)|\in \ell^2(\mathbb{N})\right\}.
\]

\begin{propo}
\label{propo:uniqueness}
For each $f\in \ell^2(\mathbb{N})$, the unique solution in $\ell_{\star}^{2}$ of the initial-value problem \eqref{initialvalueproblem} is $W_t^{(\alpha,\beta)}f$.
\end{propo}
\begin{proof}
From the estimate \eqref{ec:L2-max}, it is clear that $W_t^{(\alpha,\beta)}$ is a solution of the initial-value problem \eqref{initialvalueproblem} in the space $\ell^2_{\star}$. It remains to prove the uniqueness in this space.

Suppose that \eqref{initialvalueproblem} has two solutions in $\ell^2_\star$, say for example, $u_{1}(n,t)$ and $u_{2}(n,t)$. Note that $u(n,t)=u_{1}(n,t)-u_{2}(n,t)$ solves
\begin{equation}
\begin{cases}
\dfrac{\partial u(n,t)}{\partial t}=\mathcal{J}^{(\alpha,\beta)}u(n,t),\\[4pt]
u(n,0)=0.
\end{cases}
\end{equation}
Then, set
\begin{equation*}
\mathcal{E}(t) = \frac{1}{2} \sum_{n=0}^{\infty} (u(n,t))^{2},\quad t \ge 0.
\end{equation*}
Clearly, $\mathcal{E}(t)\geq 0$ for all $t\ge 0$. Furthermore, using \eqref{initialvalueproblem} and the definition of $\mathcal{J}^{(\alpha,\beta)}$, we have that
\[
\left|\frac{\partial}{\partial t}((u(n,t))^2)\right|\le C(|u(n,t)u(n-1,t)|+(u(n,t))^2+|u(n,t)u(n+1,t)|)
\]
and then, using that $u\in \ell^2_\star$, we can interchange the derivative and the sum. In this way,
\begin{equation*}
\frac{d}{dt}\mathcal{E}(t) = \sum_{n=0}^{\infty} u(n,t)\mathcal{J}^{(\alpha,\beta)}u(n,t)=-\sum_{n=0}^{\infty} u(n,t) \delta^{\star}{\delta} u(n,t).
\end{equation*}
Since $\delta$ and $\delta^{\star}$ are adjoint operators in $\ell^{2}(\mathbb{N})$ we deduce that
\[
\frac{d}{dt}\mathcal{E}(t)=-\sum_{n=0}^{\infty}(\delta u(n,t))^2\leq 0.
\]
Hence, since $\mathcal{E}(0)=0$, we deduce that $\mathcal{E}(t)\leq 0$ and then $u_1(n,t)=u_2(n,t)$.
\end{proof}

We note that our election here of the space $\ell_{\star}^2$ is justified in order to guarantee the derivation under the summation, but any other space which allows us to do that step can be used instead of that space.


\subsection{On the positivity of $W_t^{(\alpha,\beta)}$}
Now, we investigate the positivity of the heat semigroup $\{W_t^{(\alpha,\beta)}\}_{t\ge 0}$. In fact, we are going to prove that $W_t^{(\alpha,\beta)}f$ is non-negative provided $f$ is a non-negative sequence in $\ell^\infty(\mathbb{N})$. Note that, from the results in the previous section, we have that $W_t^{(\alpha,\beta)}$ is well defined for sequences in $\ell^{2}(\mathbb{N})$. So, we first have to extend the definition to the space $\ell^\infty(\mathbb{N})$; this extension is well defined due to the following lemma.
\begin{lem}
\label{lem:acot-pos}
Let $\alpha,\beta\ge -1/2$ and $n\not=m$, then
\begin{equation}
\label{eq:l-infty}
|K_t^{(\alpha,\beta)}(m,n)|\le C \frac{t^{1/2}}{|m-n|^2}.
\end{equation}
\end{lem}
The proof of this result will given in the last section and will be a consequence of Lemma \ref{lem:aux-parts}.

The main tool to prove the positivity of $W_t^{(\alpha,\beta)}$ is a well-known linearization formula for the product of two Jacobi polynomials due to G. Gasper. For the normalized polynomials $p_n^{(\alpha,\beta)}$, $n\geq 0$, $\alpha$, $\beta>-1$, it reads as follows
\begin{equation}
\label{eq:Gasper}
p_m^{(\alpha,\beta)}(x)p_n^{(\alpha,\beta)}(x)=\sum_{k=|m-n|}^{m+n}c(k,n,m,\alpha,\beta)p_k^{(\alpha,\beta)}(x),
\end{equation}
where the coefficients $c(k,n,m,\alpha,\beta)$ are non-negative if and only if $(\alpha,\beta)\in V$. Here we say that $(\alpha,\beta)$ belongs to the set $V$ if $\alpha,\beta>-1$, $\alpha\ge \beta$ and
\[
(\alpha+\beta+1)(\alpha+\beta+4)^2(\alpha+\beta+6)\ge (\alpha-\beta)^2((\alpha+\beta+1)^2-7(\alpha+\beta+1)-24).
\]
The previous general result was given in \cite{Gasper2}. The result in \cite{Gasper1} establishes the positivity of the coefficients $c(k,n,m,\alpha,\beta)$ under the simpler (but less general) conditions $\alpha\ge \beta$ and $\alpha+\beta\ge -1$.

\begin{thm}
\label{thm:positivity}
Let $\alpha\ge \beta\ge -1/2$. Then for each non-negative sequence $f\in \ell^\infty(\mathbb{N})$, the heat operator $W_t^{(\alpha,\beta)}f$, with $t>0$, is non-negative.
\end{thm}
\begin{proof}
The operator $W_t^{(\alpha,\beta)}$ is well defined for sequences in $\ell^\infty(\mathbb{N})$ by Lemma \ref{lem:acot-pos}. In order to prove the result it suffices to check that the kernel $K_t^{(\alpha,\beta)}$ is non-negative. By using the linearization formula \eqref{eq:Gasper}  we can express the kernel in the following way
\begin{equation}\label{eq:liner}
K_t^{(\alpha,\beta)}(m,n)=\sum_{k=|m-n|}^{m+n}c(k,n,m,\alpha,\beta)\int_{-1}^1 e^{-(1-x)t}p_k^{(\alpha,\beta)}(x)\, d\mu_{\alpha,\beta}(x),
\end{equation}
with $c(k,n,m,\alpha,\beta)\ge 0$. So, by using that
\[
h^{(\alpha,\beta)}_{t}(k):=\int_{-1}^1 e^{-(1-x)t}p_k^{(\alpha,\beta)}(x)\, d\mu_{\alpha,\beta}(x)=e^{-t}w_{k}^{(\alpha,\beta)}\int_{-1}^1 e^{xt}P_k^{(\alpha,\beta)}(x)\, d\mu_{\alpha,\beta}(x)
\]
the proof reduces to show the non-negativity of the last integral, but it is almost immediate from Rodrigues' formula and integrating $k$ times by parts. Indeed,
\begin{align*}
\int_{-1}^{1}e^{xt}P_{k}^{(\alpha,\beta)}(x)\,d\mu_{\alpha,\beta}(x)
	&=\frac{(-1)^{k}}{2^{k}k!}
		\int_{-1}^{1}e^{xt}\frac{d^{k}}{dx^{k}}\left((1-x)^{\alpha+k}(1+x)^{\beta+k}\right)\,dx \\
	&=\frac{t^{k}}{2^{k}k!}\int_{-1}^{1}e^{xt}(1-x)^{\alpha+k}(1+x)^{\beta+k}\,dx,
\end{align*}
which is clearly non-negative.

\end{proof}

It is worth pointing out that by means of \eqref{eq:Gasper} it is possible to define a positive convolution operator in $\ell^{1}(\mathbb{N})$ \cite[Corollary 1]{Gasper2}. Actually, this procedure is extensible to other orthogonal polynomials and it has been widely studied for example in~\cite{Szwarc}, where general orthogonal polynomials are considered. This convolution structure is used in \cite[eq. 12]{Bet-et-al} to study several questions related to the heat semigroup in the ultraspherical setting. Regarding the Jacobi one, the convolution operator is given by
$$
(f\ast g)(n)=\sum_{m=0}^{\infty}f(m)\tau^{(\alpha,\beta)}_{n}g(m), \quad n\in\mathbb{N},
$$
where $\tau^{(\alpha,\beta)}_{n}g(m)$ denotes the translation operator
$$
\tau^{(\alpha,\beta)}_{n}g(m)=\sum_{k=|m-n|}^{m+n}c(k,m,n,\alpha,\beta)g(k).
$$
Rewriting the equation \eqref{eq:liner} in terms of the translation operator as
$
K^{(\alpha,\beta)}_{t}(m,n)=\tau^{(\alpha,\beta)}_{n}h^{(\alpha,\beta)}_{t}(m)
$
it is straightforward to give an expression for the heat semigroup as a convolution by
$$
W^{(\alpha,\beta)}_{t}f(n)=(f\ast h^{(\alpha,\beta)}_{t})(n).
$$
However, we will not follow this approach to analyse the heat semigroup.

\begin{rem}
It is interesting to observe that the proof also works if we consider sequences in $\ell^2(\mathbb{N})$ instead of in $\ell^\infty(\mathbb{N})$. In that case, the hypotheses on the parameters $\alpha$ and $\beta$ could be improved to $(\alpha,\beta)\in V$.
\end{rem}
\section{Weighted inequalities for maximal operators}
\label{sec:max}

We devote this section to analyze some weighted inequalities for the operators $W_{\ast}^{(\alpha,\beta)}$ and $P_{\ast}^{(\alpha,\beta)}$ by using an appropriate discrete vector-valued local Calder\'on-Zygmund theory which has been developed in~\cite{Bet-et-al}. There, the interested reader can find some useful references concerning vector-valued Calder\'{o}n-Zygmund theory.

The next theorem includes mapping properties in weighted $\ell^{p}$-spaces of the heat maximal operator $W_{\ast}^{(\alpha,\beta)}$.

\begin{thm}
\label{th:max}
Let $\alpha,\beta\ge -1/2$ and consider the maximal operator $W_{\ast}^{(\alpha,\beta)}$ defined by \eqref{eq:max-heat}.
\begin{enumerate}
\item[(a)]
If $1<p<\infty$ and $w\in A_p(\mathbb{N})$, then
\begin{equation}
\label{eq:bound-max-heat}
\|W_{\ast}^{(\alpha,\beta)}f\|_{\ell^p(\mathbb{N},w)}\le C \|f\|_{\ell^p(\mathbb{N},w)},
\end{equation}
for all $f\in\ell^{p}(\mathbb{N},w)$.
\item[(b)]
If $w\in A_1(\mathbb{N})$, then
\begin{equation}
\label{eq:bound-max-heat-weak}
\|W_{\ast}^{(\alpha,\beta)}f\|_{\ell^{1,\infty}(\mathbb{N},w)}\le C \|f\|_{\ell^1(\mathbb{N},w)},
\end{equation}
for all $f\in\ell^{1}(\mathbb{N},w)$.
\end{enumerate}
\end{thm}

As an immediate consequence of Theorem \ref{th:max} and the pointwise domination
\[
P_{\ast}^{(\alpha,\beta)}f(n)\le W_{\ast}^{(\alpha,\beta)}f(n),\quad n\geq 0,\quad\alpha,\beta>-1,
\]
which follows from the definition of the Poisson operator (see~\eqref{eq:poisson}), we deduce the following result for the maximal operator of the Poisson operator $P_t^{(\alpha,\beta)}$.

\begin{cor}
\label{cor:max}
Let $\alpha,\beta\ge -1/2$ and consider the maximal operator $P_{\ast}^{(\alpha,\beta)}$ defined by \eqref{eq:max-Poisson}.
\begin{enumerate}
\item[(a)]
If $1<p<\infty$ and $w\in A_p(\mathbb{N})$, then
\begin{equation*}
\|P_{\ast}^{(\alpha,\beta)}f\|_{\ell^p(\mathbb{N},w)}\le C \|f\|_{\ell^p(\mathbb{N},w)},
\end{equation*}
for all $f\in\ell^{p}(\mathbb{N},w)$.
\item[(b)]
If $w\in A_1(\mathbb{N})$, then
\begin{equation*}
\|P_{\ast}^{(\alpha,\beta)}f\|_{\ell^{1,\infty}(\mathbb{N},w)}\le C \|f\|_{\ell^1(\mathbb{N},w)},
\end{equation*}
for all $f\in\ell^{1}(\mathbb{N},w)$.
\end{enumerate}
\end{cor}

We should remark here that the previous corollary could be stated for other subordinated semigroups. Due to \eqref{eq:poisson}, it is clear that Poisson semigroup is subordinated of the heat one, so the properties related to weighted norm inequalities for the latter could be transferred to the former (see for instance \cite[Theorem 1', p. 46]{Stein1}. Unsurprisingly this transfer property not only works for the Poisson semigroup, but also for other subordinated semigroup of the heat one. As it is explained in \cite[Ch. IX, sec. 11]{Yosida}, one possible way to construct subordinated semigroups of $W_{t}^{(\alpha,\beta)}$ is essentially by means of the positive powers of the infinitesimal generator $\mathcal{J}^{(\alpha,\beta)}$. To be more specific, the semigroup with infinitesimal generator $-(-\mathcal{J}^{(\alpha,\beta)})^{\sigma}$, where $0<\sigma<1$, is subordinated of  $W_{t}^{(\alpha,\beta)}$. In the particular case of Poisson semigroup, its infinitesimal generator is given by $-(-\mathcal{J}^{(\alpha,\beta)})^{1/2}$.

%

As we have already mentioned, the proof of Theorem \ref{th:max} relies on an appropriate local theory for discrete Banach space valued Calder\'on-Zygmund operators which is presented in \cite{Bet-et-al}. Although the $\ell^{2}$-boundedness mapping properties are derived in a similar way as in that paper and some ideas to reduce the proof of Theorem~\ref{th:max} are taken from there, our method for showing Calder\'{o}n-Zygmund estimates is, however, completely new and more straightforward.

For the reader's convenience, it is appropriate to recall some of the basic aspects of this local theory.

Suppose that $\mathbb{B}_1$ and $\mathbb{B}_2$ are Banach spaces. We denote by $\mathcal{L}(\mathbb{B}_1,\mathbb{B}_2)$ the space of bounded linear operators from $\mathbb{B}_1$ into $\mathbb{B}_2$. Let us suppose that
\[
K:(\mathbb{N}\times\mathbb{N})\setminus D \longrightarrow \mathcal{L}(\mathbb{B}_1,\mathbb{B}_2),
\]
where $D:=\{(n,n):n\in \mathbb{N}\}$, is measurable and that for certain positive constant $C$ and for each $n$, $m\in \mathbb{N}$, the following conditions hold.
\begin{enumerate}
\item[(a)] the size condition:
\[
\|K(n,m)\|_{\mathcal{L}(\mathbb{B}_1,\mathbb{B}_2)}\le \frac{C}{|n-m|},
\]
\item[(b)] the regularity properties:
\begin{enumerate}
\item[(b1)]
\[
\|K(n,m)-K(l,m)\|_{\mathcal{L}(\mathbb{B}_1,\mathbb{B}_2)}\le C \frac{|n-l|}{|n-m|^2},\quad |n-m|>2|n-l|, \frac{m}{2}\le n,l\le \frac{3m}{2},
\]
\item[(b2)]
\[
\|K(m,n)-K(m,l)\|_{\mathcal{L}(\mathbb{B}_1,\mathbb{B}_2)}\le C \frac{|n-l|}{|n-m|^2},\quad |n-m|>2|n-l|, \frac{m}{2}\le n,l\le \frac{3m}{2}.
\]
\end{enumerate}
\end{enumerate}
A kernel $K$ satisfying conditions (a) and (b) is called a local $\mathcal{L}(\mathbb{B}_1,\mathbb{B}_2)$-standard kernel. For a Banach space $\mathbb{B}$ and a weight $w=\{w(n)\}_{n\ge 0}$, we consider the space
\[
\ell^{p}_{\mathbb{B}}(\mathbb{N},w)=\left\{ \text{$\mathbb{B}$-valued sequences } f=\{f(n)\}_{n\ge 0}: \{\|f(n)\|_{\mathbb{B}}\}_{n\ge 0}\in \ell^p(\mathbb{N},w)\right\}
\]
for $1\le p<\infty$, and
\[
\ell^{1,\infty}_{\mathbb{B}}(\mathbb{N},w)=\left\{ \text{$\mathbb{B}$-valued sequences } f=\{f(n)\}_{n\ge 0}: \{\|f(n)\|_{\mathbb{B}}\}_{n\ge 0}\in \ell^{1,\infty}(\mathbb{N},w)\right\}.
\]
As usual, we simply write $\ell_{\mathbb{B}}^r(\mathbb{N})$ and $\ell^{1,\infty}_{\mathbb{B}}(\mathbb{N})$ when $w(n)=1$ for all $n\in \mathbb{N}$. Also, by $\mathbb{B}_0^{\mathbb{N}}$ we represent the space of $\mathbb{B}$-valued sequences $f=\{f(n)\}_{n\ge 0}$ such that $f(n)=0$, with $n>j$, for some $j\in \mathbb{N}$.
\begin{thm}[Theorem 2.1 in \cite{Bet-et-al}]
\label{thm:CZ}
Let $\mathbb{B}_1$ and $\mathbb{B}_2$ be Banach spaces. Suppose that $T$
is a linear and bounded operator from $\ell_{\mathbb{B}_1}^r(\mathbb{N})$ into $\ell_{\mathbb{B}_2}^r(\mathbb{N})$, for some $1<r<\infty$, and such that there exists a local $\mathcal{L}(\mathbb{B}_1,\mathbb{B}_2)$-standard kernel $K$ such that, for every sequence $f\in (\mathbb{B}_1)_0^{\mathbb{N}}$,
\[
Tf(n)=\sum_{m=0}^{\infty}K(n,m)\cdot f(m),
\]
for every $n\in \mathbb{N}$ such that $f(n)=0$. Then,
\begin{enumerate}
\item[(i)] for every $1< p <\infty$ and $w\in A_p(\mathbb{N})$ the operator $T$
can be extended from $\ell_{\mathbb{B}_1}^r(\mathbb{N})\cap \ell_{\mathbb{B}_1}^p(\mathbb{N},w)$
to $\ell_{\mathbb{B}_1}^p(\mathbb{N},w)$ as a bounded operator from $\ell_{\mathbb{B}_1}^p(\mathbb{N},w)$ into $\ell_{\mathbb{B}_2}^p(\mathbb{N},w)$.

\item[(ii)] for every $w\in A_1(\mathbb{N})$ the operator $T$
can be extended from $\ell_{\mathbb{B}_1}^r(\mathbb{N})\cap \ell_{\mathbb{B}_1}^1(\mathbb{N},w)$
to $\ell_{\mathbb{B}_1}^1(\mathbb{N},w)$ as a bounded operator from $\ell_{\mathbb{B}_1}^p(\mathbb{N},w)$ into $\ell_{\mathbb{B}_2}^{1,\infty}(\mathbb{N},w)$.
\end{enumerate}
\end{thm}
Now, set $\mathbb{B}=L^{\infty}(0,\infty)$. In order to prove Theorem~\ref{th:max}, we observe first that the operator
\begin{equation*}
\begin{array}{ccccc}
T & : & \ell^{2}(\mathbb{N}) & \longrightarrow & \ell_{\mathbb{B}}^{2}(\mathbb{N})\\
& & f & \longmapsto & T f(n,t) := W_{t}^{(\alpha,\beta)}f(n),
\end{array}
\end{equation*}
is bounded from $\ell^{2}(\mathbb{N})$ into $\ell_{\mathbb{B}}^{2}(\mathbb{N})$. Indeed, it is a consequence of the $\ell^{2}$-boundedness of the heat maximal operator (see~\eqref{ec:L2-max}) as it happened in the case of ultraspherical expansions in~\cite{Bet-et-al}.

Note that in order to obtain the regularity properties (b1) and (b2) for the kernel $K_t^{(\alpha,\beta)}$ it suffices to prove the inequality
\begin{equation}\label{Calderon-Zygmundn+1}
\| K_{t}^{(\alpha,\beta)}(n+1,m) - K_{t}^{(\alpha,\beta)}(n,m)\|_{L^{\infty}(0,\infty)} \leq \frac{C}{|n-m|^{2}},
\end{equation}
for $n,m\in\mathbb{N}$, $n\neq m$, and $\frac{m}{2}\leq n\leq \frac{3m}{2}$. The proof of this fact is based on the ideas of \cite[p. 14]{Bet-et-al} and it actually works when the norm comes from a general Banach space $\mathbb{B}$.

Let us see that \eqref{Calderon-Zygmundn+1} implies (b1) (the proof for (b2) is analogous). If $n=l$ the conclusion follows readily. Let us suppose that $n<l$. By the triangle inequality, we obtain
\begin{multline*}
\|K_t^{(\alpha,\beta)}(n,m)-K_t^{(\alpha,\beta)}(l,m)\|_{L^\infty(0,\infty)}\\ \leq \sum_{j=0}^{l-n-1} \|K_t^{(\alpha,\beta)}(n+j,m)-K_t^{(\alpha,\beta)}(n+j+1,m)\|_{L^\infty(0,\infty)}.
\end{multline*}
If $n>m$ we apply \eqref{Calderon-Zygmundn+1} to get the desired estimate. When $n<m$ we apply \eqref{Calderon-Zygmundn+1} and then use that $|n-m|>2|n-l|$ so the result follows. The case $n>l$ is similar and we omit the details.

The following results, whose proofs are postponed to the next section, together with the above-mentioned $\ell^{2}$-boundedness imply Theorem~\ref{th:max}.
\begin{lem}
\label{lem:heat-bound}
Let $\alpha,\beta\ge-1/2$ and $t \ge 0$.  If $n$, $m\in\mathbb{N}$, $n\not= m$, then
\begin{equation}
\label{ec:heat-bound}
|K_t^{(\alpha,\beta)}(n,m)|\le \frac{C}{|n-m|}.
\end{equation}
\end{lem}

\begin{lem}
\label{lem:heat-bound-diff}
Let $\alpha,\beta\ge-1/2$ and $t \ge 0$. If $n$, $m\in\mathbb{N}$, $m\not=n$ and $m/2\le n \le 3m/2$, then
\begin{equation}
\label{ec:heat-bound-diff}
\sup_{t>0}|K_t^{(\alpha,\beta)}(n+1,m)-K_t^{(\alpha,\beta)}(n,m)|\le \frac{C}{|n-m|^{2}}.
\end{equation}
\end{lem}

\section{Technical results}
\label{sec:tech}

We adopt the following notation
\[
\mathfrak{I}_t^{(a,b,A,B,\alpha,\beta)}(n,m)=\int_{-1}^{1}e^{-t(1-x)}P_n^{(a,b)}(x)P_m^{(A,B)}(x)(1-x)^\alpha (1+x)^\beta \, dx,
\]
where $t \ge 0$, $a,b,A,B,\alpha,\beta>-1$, and $n,m\ge 0$.

We start this section giving a technical lemma related to the family of integrals $\mathfrak{I}_t^{(a,b,A,B,\alpha,\beta)}(n,m)$  which we will use to prove Lemma \ref{lem:acot-pos} and the required Calder\'{o}n-Zygmund bounds contained in Lemmas \ref{lem:heat-bound} and \ref{lem:heat-bound-diff}.
\begin{lem}
\label{lem:aux-parts}
Let $n$, $m\in\mathbb{N}$ and $a$, $b$, $A$, $B$, $\alpha$, $\beta>-1$ such that $n+a+b+1\neq 0$, $m+A+B+1\neq 0$, and $n(n+a+b+1)\neq m(m+A+B+1)$.

\begin{enumerate}
\item[a)] If $n$, $m\neq 0$, we have that
\begin{equation*}
\begin{aligned}
\mathfrak{I}_{t}^{(a,b,A,B,\alpha,\beta)}(n,m) &= \frac{(n+a+b+1)(m+A+B+1)}{2(n(n+a+b+1)-m(m+A+B+1))} \\&\times \Biggl( \frac{t}{m+A+B+1}\mathfrak{I}_{t}^{(a+1,b+1A,B,\alpha+1,\beta+1)}(n-1,m) \\&- \frac{\alpha-a}{m+A+B+1}\mathfrak{I}_{t}^{(a+1,b+1,A,B,\alpha,\beta+1)}(n-1,m) \\&+ \frac{\beta-b}{m+A+B+1}\mathfrak{I}_{t}^{(a+1,b+1,A,B,\alpha+1,\beta)}(n-1,m) \\&- \frac{t}{n+a+b+1}\mathfrak{I}_{t}^{(a,b,A+1,B+1,\alpha+1,\beta+1)}(n,m-1) \\&+ \frac{\alpha-A}{n+a+b+1}\mathfrak{I}_{t}^{(a,b,A+1,B+1,\alpha,\beta+1)}(n,m-1) \\&- \frac{\beta-B}{n+a+b+1}\mathfrak{I}_{t}^{(a,b,A+1,B+1,\alpha+1,\beta)}(n,m-1)  \Biggr).
\end{aligned}
\end{equation*}

\item[b)] If $n=0$ and $m\in\mathbb{N}$,
\begin{equation*}
\begin{aligned}
\mathfrak{I}_{t}^{(a,b,A,B,\alpha,\beta)}(0,m) &= \frac{t}{2m} \mathfrak{I}_{t}^{(a,b,A+1,B+1,\alpha+1,\beta+1)}(0,m-1) \\&- \frac{\alpha-A}{2m} \mathfrak{I}_{t}^{(a,b,A+1,B+1,\alpha,\beta+1)}(0,m-1) \\&+ \frac{\beta-B}{2m} \mathfrak{I}_{t}^{(a,b,A+1,B+1,\alpha+1,\beta)}(0,m-1).
\end{aligned}
\end{equation*}

\item[c)] If $n\in\mathbb{N}$ and $m=0$,
\begin{equation*}
\begin{aligned}
\mathfrak{I}_{t}^{(a,b,A,B,\alpha,\beta)}(n,0) &= \frac{t}{2n} \mathfrak{I}_{t}^{(a+1,b+1,A,B,\alpha+1,\beta+1)}(n-1,0) \\&- \frac{\alpha-a}{2n} \mathfrak{I}_{t}^{(a+1,b+1,A,B,\alpha,\beta+1)}(n-1,0) \\&+ \frac{\beta-b}{2n} \mathfrak{I}_{t}^{(a+1,b+1,A,B,\alpha+1,\beta)}(n-1,0).
\end{aligned}
\end{equation*}
\end{enumerate}

\end{lem}

\begin{proof}
First, we prove case a). By using the identities (see \cite[p. 63, eq.~(4.21.7)]{Szego})
\[
\frac{d}{dx}P^{(\alpha,\beta)}_{n}(x)=\frac{1}{2}(n+\alpha+\beta+1)P^{(\alpha+1,\beta+1)}_{n-1}(x)
\]
and (see~\cite[p. 94, eq.~(4.10.1)]{Szego} or \cite[p.~446, eq.~(18.9.16)]{NIST})
\[
\frac{d}{dx}\left((1-x)^{\alpha}(1+x)^{\beta}P^{(\alpha,\beta)}_{n}(x)\right)=-2(n+1)(1-x)^{\alpha-1}(1+x)^{\beta-1}P^{(\alpha-1,\beta-1)}_{n+1}(x),
\]
and applying integration by parts we have
\begin{multline*}
\mathfrak{I}_t^{(a,b,A,B,\alpha,\beta)}(n,m)\\
\begin{aligned}
&=\frac{-1}{2n}\int_{-1}^{1}e^{-t(1-x)}\frac{d}{dx}\left(P_{n-1}^{(a+1,b+1)}(x)(1-x)^{a+1}(1+x)^{b+1}\right)P_m^{(A,B)}(x)\\
&\kern10pt\times(1-x)^{\alpha-a} (1+x)^{\beta-b} \, dx\\&=\frac{t}{2n}\mathfrak{I}_t^{(a+1,b+1,A,B,\alpha+1,\beta+1)}(n-1,m)\\
&\kern10pt+\frac{m+A+B+1}{4n}\mathfrak{I}_t^{(a+1,b+1,A+1,B+1,\alpha+1,\beta+1)}(n-1,m-1)\\
&\kern10pt-\frac{\alpha-a}{2n}\mathfrak{I}_t^{(a+1,b+1,A,B,\alpha,\beta+1)}(n-1,m)\\
&\kern10pt+\frac{\beta-b}{2n}\mathfrak{I}_t^{(a+1,b+1,A,B,\alpha+1,\beta)}(n-1,m).
\end{aligned}
\end{multline*}
In a similar way, we obtain that
\begin{multline*}
\mathfrak{I}_t^{(a+1,b+1,A+1,B+1,\alpha+1,\beta+1)}(n-1,m-1)\\
\begin{aligned}
&=\frac{2}{n+a+b+1} \int_{-1}^{1}e^{-t(1-x)}\frac{d}{dx}\Big(P_{n}^{(a,b)}(x)\Big)\\
&\kern10pt \times\Big( P_{m-1}^{(A+1,B+1)}(x)(1-x)^{A+1}(1+x)^{B+1}\Big) (1-x)^{\alpha-A} (1+x)^{\beta-B} \, dx\\
&= \frac{-2t}{n+a+b+1}\mathfrak{I}_t^{(a,b,A+1,B+1,\alpha+1,\beta+1)}(n,m-1)\\
&\kern10pt +\frac{4m}{n+a+b+1}\mathfrak{I}_t^{(a,b,A,B,\alpha,\beta)}(n,m)\\
&\kern10pt +\frac{2(\alpha-A)}{n+a+b+1}\mathfrak{I}_t^{(a,b,A+1,B+1,\alpha,\beta+1)}(n,m-1)\\
&\kern10pt -\frac{2(\beta-B)}{n+a+b+1}\mathfrak{I}_t^{(a,b,A+1,B+1,\alpha+1,\beta)}(n,m-1)
\end{aligned}
\end{multline*}
and the result follows.

For cases b) and c) write
\begin{multline*}
\mathfrak{I}_{t}^{(a,b,A,B,\alpha,\beta)}(0,m)\\
\begin{aligned}
&= -\frac{1}{2m} \int_{-1}^{1} e^{-t(1-x)}\frac{d}{dx}\left((1-x)^{A+1}(1+x)^{B+1}P_{m-1}^{(A+1,B+1)}(x)\right)\\
&\kern10pt\times(1-x)^{\alpha-A}(1+x)^{\beta-B}\,dx
\end{aligned}
\end{multline*}
in the former and
\begin{multline*}
\mathfrak{I}_{t}^{(a,b,A,B,\alpha,\beta)}(n,0)\\
\begin{aligned}
&= -\frac{1}{2n} \int_{-1}^{1} e^{-t(1-x)}\frac{d}{dx}\left((1-x)^{a+1}(1+x)^{b+1}P_{n-1}^{(a+1,b+1)}(x)\right)\\
&\kern10pt\times(1-x)^{\alpha-a}(1+x)^{\beta-b}\,dx
\end{aligned}
\end{multline*}
in the latter and integrate by parts.
\end{proof}

\begin{proof}[Proof of Lemma \ref{lem:acot-pos}]
We only check the cases $n,m\ge 2$, with $n\not=m$, by using a) in Lemma \ref{lem:aux-parts}. The remaining cases can be obtained from b) and c) in the same lemma. Taking $a=A=\alpha$ and $b=B=\beta$ in Lemma \ref{lem:aux-parts} case a) and noting that
\[
n(n+\alpha+\beta+1)-m(m+\alpha+\beta+1)=(n-m)(n+m+\alpha+\beta+1)
\]
we have
\begin{multline}
\label{eq:heat-1}
|K_t^{(\alpha,\beta)}(n,m)|\le \frac{C t}{|n-m|}w_{n}^{(\alpha,\beta)}w_{m}^{(\alpha,\beta)} \Bigg(\left|\mathfrak{I}_t^{(\alpha+1,\beta+1,\alpha,\beta,\alpha+1,\beta+1)}(n-1,m)\right|
\\+\left|\mathfrak{I}_t^{(\alpha,\beta,\alpha+1,\beta+1,\alpha+1,\beta+1)}(n,m-1)\right|\Bigg).
\end{multline}
Lemma \ref{lem:aux-parts} gives that
\begin{align*}
\left|\mathfrak{I}_t^{(\alpha+1,\beta+1,\alpha,\beta,\alpha+1,\beta+1)}(n-1,m)\right|&\le \frac{C}{|n-m|}\Bigg(t|\mathfrak{I}^{(\alpha+2,\beta+2,\alpha,\beta,\alpha+2,\beta+2)}(n-2,m)|\\
&\kern10 pt +t|\mathfrak{I}^{(\alpha+1,\beta+1,\alpha+1,\beta+1,\alpha+2,\beta+2)}(n-1,m-1)|
\\&\kern10 pt + |\mathfrak{I}^{(\alpha+1,\beta+1,\alpha+1,\beta+1,\alpha+1,\beta+2)}(n-1,m-1)|
\\&\kern10 pt +|\mathfrak{I}^{(\alpha+1,\beta+1,\alpha+1,\beta+1,\alpha+2,\beta+1)}(n-1,m-1)|
\Bigg)
\end{align*}
Now, with the uniform bound
\begin{equation}
\label{eq:unif}
|p_n^{(a,b)}(x)|\le C (1-x)^{-a/2-1/4}(1+x)^{-b/2-1/4},\qquad a,b\geq-\frac{1}{2},
\end{equation}
which arises from \cite[eq. 7.32.6]{Szego}, taking into account the asymptotic behaviour
\begin{equation}
\label{eq:asym-wn}
w_n^{(a,b)}\sim n^{1/2}, \qquad a,b \ge -\frac{1}{2},
\end{equation}
and the bound
\begin{align*}
\int_{-1}^{0}e^{-t(1-x)}(1-x)^{-1/2}\,dx&=\int_{0}^{1}e^{-t(1+x)}(1+x)^{-1/2}\,dx
\\&\le \int_{0}^{1}e^{-t(1-x)}(1-x)^{-1/2}\,dx,
\end{align*}
we obtain that
\begin{multline*}
\left|\mathfrak{I}_t^{(\alpha+1,\beta+1,\alpha,\beta,\alpha+1,\beta+1)}(n-1,m)\right|\\
\begin{aligned}
&\le \frac{C}{\sqrt{nm}|n-m|}\Bigg(t\int_{-1}^{1}e^{-t(1-x)}(1-x)^{1/2}\,dx+\int_{0}^{1}e^{-t(1-x)}(1-x)^{-1/2}\,dx\Bigg)\\
&\le \frac{C}{\sqrt{nm}|n-m|}\Bigg(t^{1/2}\int_{-1}^{1}e^{-t(1-x)/2}\,dx+t^{-1/2}\int_{0}^{\infty}e^{-s}s^{-1/2}\,ds\Bigg)\\
&\le \frac{Ct^{-1/2}}{\sqrt{nm}|n-m|}.
\end{aligned}
\end{multline*}
Following the same procedure, we deduce that
\[
\left|\mathfrak{I}_t^{(\alpha,\beta,\alpha+1,\beta+1,\alpha+1,\beta+1)}(n,m-1)\right|\le \frac{Ct^{-1/2}}{\sqrt{nm}|n-m|}.
\]
Then, from \eqref{eq:heat-1}, the estimate \eqref{eq:l-infty} follows.
\end{proof}

\begin{proof}[Proof of Lemma~\ref{lem:heat-bound}]

For the cases $n,m\ge 1$, with $n\not=m$, the result follows from \eqref{eq:heat-1}, \eqref{eq:unif}, and \eqref{eq:asym-wn}. The estimate \eqref{ec:heat-bound} for the remaining cases is a consequence of b) and c) in Lemma \ref{lem:aux-parts}.

\end{proof}

\begin{proof}[Proof of Lemma~\ref{lem:heat-bound-diff}]

Note that the only possibility for $m$ satisfying the assumptions when $n=1$ is $m=2$. So let us suppose firstly that $n$ is not equal to 1 and $m$ not equal to 2 simultaneously.

We begin using the relation (see \cite[p. 71, eq.~(4.5.4)]{Szego})
\[
\frac{2n+\alpha+\beta+2}{2}(1-x)P_n^{(\alpha+1,\beta)}(x)=(n+\alpha+1)P_n^{(\alpha,\beta)}(x)-(n+1)P_{n+1}^{(\alpha,\beta)}(x),
\]
to get
\begin{align*}
p_n^{(\alpha,\beta)}(x)-p_{n+1}^{(\alpha,\beta)}(x)
&=\left(1-\frac{w_{n+1}^{(\alpha,\beta)}}{w_n^{(\alpha,\beta)}}\right)p_n^{(\alpha,\beta)}(x)
+w_{n+1}^{(\alpha,\beta)}\left(\frac{p_n^{(\alpha,\beta)}(x)}{w_n^{(\alpha,\beta)}}
-\frac{p_{n+1}^{(\alpha,\beta)}(x)}{w_{n+1}^{(\alpha,\beta)}}\right)\\&=\left(1-\frac{w_{n+1}^{(\alpha,\beta)}}{w_n^{(\alpha,\beta)}}\right)p_n^{(\alpha,\beta)}(x)
-\frac{\alpha}{n+1}\frac{w_{n+1}^{(\alpha,\beta)}}{w_n^{(\alpha,\beta)}}p_{n}^{(\alpha,\beta)}(x)
\\&\kern25pt
+\frac{2n+\alpha+\beta+2}{2(n+1)}\frac{w_{n+1}^{(\alpha,\beta)}}{w_n^{(\alpha+1,\beta)}}(1-x)p_{n}^{(\alpha+1,\beta)}(x).
\end{align*}
Therefore,
\begin{multline*}
K_t^{(\alpha,\beta)}(n,m)-K_t^{(\alpha,\beta)}(n+1,m)\\
\begin{aligned}
&=\left(1-\frac{w_{n+1}^{(\alpha,\beta)}}{w_n^{(\alpha,\beta)}}\right)K_t^{(\alpha,\beta)}(n,m)-
\frac{\alpha}{n+1}\frac{w_{n+1}^{(\alpha,\beta)}}{w_n^{(\alpha,\beta)}}K_t^{(\alpha,\beta)}(n,m)\\
&\kern25pt +\frac{2n+\alpha+\beta+2}{2(n+1)}\frac{w_{n+1}^{(\alpha,\beta)}}{w_n^{(\alpha+1,\beta)}}D^{(\alpha,\beta)}_t(n,m),
\end{aligned}
\end{multline*}
with
\[
D_t^{(\alpha,\beta)}(n,m)=w_{n}^{(\alpha+1,\beta)}w_m^{(\alpha,\beta)}\mathfrak{I}_t^{(\alpha+1,\beta,\alpha,\beta,\alpha+1,\beta)}(n,m).
\]
Now, the limit
$$
\lim_{n\to\infty}n\left(\frac{w_{n+1}^{(\alpha,\beta)}}{w_{n}^{(\alpha,\beta)}}-1\right)=\frac{1}{2},
$$
which is a consequence of \eqref{eq:asym-wn}, gives us the estimate
\[
\left|1-\frac{w_{n+1}^{(\alpha,\beta)}}{w_{n}^{(\alpha,\beta)}}\right|\leq\frac{C}{n}.
\]
Last inequality is used together with Lemma~\ref{lem:heat-bound} to obtain
\begin{equation*}
\sup_{t>0} \left|K_t^{(\alpha,\beta)}(n,m)-K_t^{(\alpha,\beta)}(n+1,m)\right| \leq \frac{C}{|n-m|^{2}} + C \sup_{t>0} \left|D_{t}^{(\alpha,\beta)}(n,m)\right|.
\end{equation*}
So the study reduces to prove that
\begin{equation}\label{eq:Dt}
\sup_{t>0} \left|D_{t}^{(\alpha,\beta)}(n,m)\right| \leq \frac{C}{|n-m|^{2}}.
\end{equation}

First, if we apply Lemma \ref{lem:aux-parts}, we obtain that
\begin{multline*}
\left|\mathfrak{I}_t^{(\alpha+1,\beta,\alpha,\beta,\alpha+1,\beta)}(n,m)\right|\le \frac{C}{|n-m|}\Bigg(t\left|\mathfrak{I}_t^{(\alpha+2,\beta+1,\alpha,\beta,\alpha+2,\beta+1)}(n-1,m)\right|\\
\begin{aligned}
&+t\left|\mathfrak{I}_t^{(\alpha+1,\beta,\alpha+1,\beta+1,\alpha+2,\beta+1)}(n,m-1)\right|\\
&+\left|\mathfrak{I}_t^{(\alpha+1,\beta,\alpha+1,\beta+1,\alpha+1,\beta+1)}(n,m-1)\right|\Bigg)
\end{aligned}
\end{multline*}
Now, applying again Lemma \ref{lem:aux-parts} to each term on the right-hand side of the previous inequality we have
\begin{multline*}
t\left|\mathfrak{I}_t^{(\alpha+2,\beta+1,\alpha,\beta,\alpha+2,\beta+1)}(n-1,m)\right|\le \frac{C}{|n-m|}\\
\begin{aligned}
&\times \Bigg(t^2\left|\mathfrak{I}_t^{(\alpha+3,\beta+2,\alpha,\beta,\alpha+3,\beta+2)}(n-2,m)\right|\\
&+t^2\left|\mathfrak{I}_t^{(\alpha+2,\beta+1,\alpha+1,\beta+1,\alpha+3,\beta+2)}(n-1,m-1)\right|\\
&+t\left|\mathfrak{I}_t^{(\alpha+2,\beta+1,\alpha+1,\beta+1,\alpha+2,\beta+2)}(n-1,m-1)\right|\\
&+t\left|\mathfrak{I}_t^{(\alpha+2,\beta+1,\alpha+1,\beta+1,\alpha+3,\beta+1)}(n-1,m-1)\right|\Bigg),
\end{aligned}
\end{multline*}
\begin{multline*}
t\left|\mathfrak{I}_t^{(\alpha+1,\beta,\alpha+1,\beta+1,\alpha+2,\beta+1)}(n,m-1)\right|\le \frac{C}{|n-m|}\\
\begin{aligned}
&\times \Bigg(t^2\left|\mathfrak{I}_t^{(\alpha+2,\beta+1,\alpha+1,\beta+1,\alpha+3,\beta+2)}(n-1,m-1)\right|\\
&\kern 10pt +t^2\left|\mathfrak{I}_t^{(\alpha+1,\beta,\alpha+2,\beta+2,\alpha+3,\beta+2)}(n,m-2)\right|\\
&\kern 10pt +t\left|\mathfrak{I}_t^{(\alpha+2,\beta+1,\alpha+1,\beta+1,\alpha+2,\beta+2)}(n-1,m-1)\right|
\\ &\kern 10pt
+t\left|\mathfrak{I}_t^{(\alpha+1,\beta,\alpha+2,\beta+2,\alpha+2,\beta+2)}(n,m-2)\right|\\
&\kern 10pt
+t\left|\mathfrak{I}_t^{(\alpha+2,\beta+1,\alpha+1,\beta+1,\alpha+3,\beta+1)}(n-1,m-1)\right|\Bigg),
\end{aligned}
\end{multline*}
and
\begin{multline*}
\left|\mathfrak{I}_t^{(\alpha+1,\beta,\alpha+1,\beta+1,\alpha+1,\beta+1)}(n,m-1)\right|\le \frac{C}{|n-m|}\\
\begin{aligned}
&\times \Bigg(t\left|\mathfrak{I}_t^{(\alpha+2,\beta+1,\alpha+1,\beta+1,\alpha+2,\beta+2)}(n-1,m-1)\right|\\
& \kern10pt +t\left|\mathfrak{I}_t^{(\alpha+1,\beta,\alpha+2,\beta+2,\alpha+2,\beta+2)}(n,m-2)\right|\\
& \kern10pt +\left|\mathfrak{I}_t^{(\alpha+2,\beta+1,\alpha+1,\beta+1,\alpha+2,\beta+1)}(n-1,m-1)\right|\Bigg).
\end{aligned}
\end{multline*}
Finally, by using the bound \eqref{eq:unif} and \eqref{eq:asym-wn} we conclude that
\begin{multline*}
\left|D^{(\alpha,\beta)}_t(n,m)\right|\le C |n-m|^{-2}\\
\begin{aligned}
&\times\Bigg( t^2\int_{-1}^{1}e^{-t(1-x)}(1-x)(1+x)^{1/2}\, dx+
t\int_{-1}^{1}e^{-t(1-x)}(1+x)^{1/2}\, dx\\
&\kern12pt+t\int_{-1}^{1}e^{-t(1-x)}(1-x)(1+x)^{-1/2}\, dx+\int_{-1}^{1}e^{-t(1-x)}(1+x)^{-1/2}\, dx\Bigg)\\
&\le C |n-m|^{-2}
\end{aligned}
\end{multline*}
and the proof of \eqref{eq:Dt} is completed.

For the singular case $n=1$ and $m=2$ we follow a similar argument as above but using Lemma~\ref{lem:aux-parts} cases a) and b) to get the result.
\end{proof}

\paragraph{\textbf{Acknowledgement.}} The authors would like to thank the anonymous reviewers for their helpful and constructive comments that greatly contributed to improving the final version of the paper.


\end{document}